\documentclass[12pt]{amsart}
\usepackage[top=4.2cm, bottom=4cm, left=2.4cm, right=2.4cm]{geometry}
\usepackage[utf8]{inputenc}
\usepackage[USenglish]{babel}
\usepackage[T1]{fontenc} 
\usepackage{mathrsfs}
\usepackage{array}
\usepackage{mathtools}
\usepackage{amsmath}
\usepackage{amssymb,epsfig}
\usepackage{caption}
\usepackage{subcaption}

\usepackage{amsthm}
\usepackage[absolute]{textpos}
\usepackage{mathtools,emptypage}

\usepackage[bookmarks=true]{hyperref}
\usepackage{xcolor}
\hypersetup{
    colorlinks,
    linkcolor={red!50!black},
    citecolor={blue!50!black},
    urlcolor={blue!80!black},
}

\usepackage{todonotes}

\mathtoolsset{showonlyrefs} 

\newcommand{\width}{{\rm width}}

\newcommand{\R}{\mathbb{R}}

\newcommand{\sfd}{{\sf d}}

\renewcommand{\d}{{\mathrm d}}

\newcommand{\restr}[1]{\lower3pt\hbox{$|_{#1}$}}

\newcommand{\limi}{\varliminf}
\newcommand{\lims}{\varlimsup}
                          
\newcommand{\X}{{\rm X}}
\newcommand{\Y}{{\rm Y}}

\newcommand{\lip}{{\rm lip}}
\newcommand{\Lip}{{\rm Lip}}
\newcommand{\loc}{{\rm loc}}

\newcommand{\mm}{\mathfrak m}

\renewcommand{\SS}{{\rm SS}}
\newcommand{\SR}{{\rm SR}}

\newcommand{\pos}{{\rm pos}}

\makeatletter
\newcommand{\mytag}[2]{%
  \text{#1}%
  \@bsphack
  \begingroup
    \@onelevel@sanitize\@currentlabelname
    \edef\@currentlabelname{%
      \expandafter\strip@period\@currentlabelname\relax.\relax\@@@%
    }%
    \protected@write\@auxout{}{%
      \string\newlabel{#2}{%
        {\color{black}#1}%
        {\thepage}%
        {\@currentlabelname}%
        {\@currentHref}{}%
      }%
    }%
  \endgroup
  \@esphack
}
\makeatother

\def\Xint#1{\mathchoice
{\XXint\displaystyle\textstyle{#1}}%
{\XXint\textstyle\scriptstyle{#1}}%
{\XXint\scriptstyle\scriptscriptstyle{#1}}%
{\XXint\scriptscriptstyle\scriptscriptstyle{#1}}%
\!\int}
\def\XXint#1#2#3{{\setbox0=\hbox{$#1{#2#3}{\int}$ }
\vcenter{\hbox{$#2#3$ }}\kern-.6\wd0}}

\def\dashint{\Xint-}

\allowdisplaybreaks

\newtheorem{theorem}{Theorem}[section]

\newtheorem{proposition}[theorem]{Proposition}
\theoremstyle{definition}
\newtheorem{definition}[theorem]{Definition}

\newtheorem{example}[theorem]{Example}

\newcounter{Counter}

\newcommand{\Int}{\textup{int}}

\newcommand{\codH}[1]{\mathcal{H}^{{\rm cod-}{#1}}}

\title[Geometric characterizations of ${\sf PI}$ spaces]{Geometric characterizations of ${\sf PI}$ spaces: an overview of some modern techniques}

\author[Emanuele Caputo]{Emanuele Caputo}\address[Emanuele Caputo]{Mathematics Institute, Zeeman Building, University of Warwick, Coventry, CV4 7AL, United Kingdom}\email{emanuele.caputo@warwick.ac.uk}


\begin{document}

\begin{abstract}
    We survey recent results on the study of metric measure spaces satisfying a Poincar\'{e} inequality.
    We overview recent characterizations in terms of objects of dimension 1, such as pencil of curves, modulus estimates and obstacle-avoidance principles.
    Then, we turn our attention to characterizations in terms of objects of codimension 1, such as relative isoperimetric inequalities and separating sets, the last one obtained in collaboration with N. Cavallucci in \cite{CaputoCavallucci2024}. We propose a strategy to provide examples using our characterization in the toy-model of the Euclidean case. We also discuss a more geometric relation between separating sets and obstacle-avoidance principles, obtained in \cite{CaputoCavallucci2024II}.Finally, we recall some open questions in the field.
\end{abstract}

\maketitle


\section{Introduction}

In recent works with N. Cavallucci \cite{CaputoCavallucci2024,CaputoCavallucci2024II, CaputoCavallucci2024III}, we obtained new characterizations of doubling metric measure spaces satisfying the $1$-Poincar\'{e} inequality in terms of a lower bound on the total amount of suitable energies of boundaries of separating sets.

This survey has a dual goal. On one side, we present our formalism together with the discussion of the main results. On the other side, we review recent techniques in metric analysis related to the study of metric measure spaces that are doubling and satisfy a $1$-Poincar\'e inequality, called ${\sf PI}$ spaces. We assume basic familiarity with metric analysis and geometric measure theory.

A metric measure space is a triple $(\X,\sfd,\mm)$, where $(\X,\sfd)$ is a complete and separable metric space and $\mm$ is a nonnegative Borel measure that is finite on bounded sets and such that $\mm(\X)>0$. The underlying standing assumption is that the measure $\mm$ is doubling, i.e. there exists a constant $C_D \ge 1$ such that 
\begin{equation*}
    0<\mm(B_{2r}(x))\le C_D \mm(B_r(x))< \infty \quad\text{for every }x \in \X,\, r>0.
\end{equation*}

Here we denote $B_r(x):=\{ y \in \X:\, \sfd(y,x)< r \}$.
We introduce the notion of a ${\sf PI}$ space. To formulate this condition, we need the definition of local Lipschitz constant, see Section \ref{sec:lipschitz_functions}.

\begin{definition}[${\sf PI}$ space]
    A metric measure space $(\X,\sfd,\mm)$ satisfies a (weak) $1$-Poincar\'{e} inequality if there exists $C_{P} \ge 1$ and $\lambda \ge 1$ such that 
    \begin{equation}
    \label{pPoincare-lip}
        \dashint_{B_r(x)} \left|u - \dashint_{B_r(x)}u\,\d \mm\right|\,\d \mm \le C_P r\, \dashint_{B_{\lambda r}(x)} \lip u\,\d \mm
    \end{equation}
    for every $u \in {\rm Lip}(\X)$.
    A metric measure space is called a ${\sf PI}$ space if it is doubling and satisfies a $1$-Poincar\'{e} inequality.
\end{definition}
The Poincar\'{e} inequality can be equivalently formulated by means of Sobolev functions, replacing the class of Lipschitz functions and with a suitable replacement of the local Lipschitz constant, but we avoid their usage in this work.

We restrict our attention from the class of all metric measure spaces to the one of ${\sf PI}$ spaces for several reasons. A non-exhaustive list is given by the following one.
\begin{itemize}
    \item \emph{A theory of differentiation of Lipschitz functions}. 
    Rademacher's theorem is a classical theorem in geometric measure theory which states that every Lipschitz function $f\colon \mathbb{R}^d \to \mathbb{R}$ is differentiable a.e.\ with respect to the Lebesgue measure.
    The study of differentiability of Lipschitz functions in a more general base space (that replaces $\mathbb{R}^d$ in the previous statement) is still a very active area of research.
    Cheeger proved in \cite{Cheeger99} that ${\sf PI}$ spaces give a rich class of examples where a version of Rademacher's theorem holds, for a suitable notion of differential of Lipschitz functions.
    
    Given $(\X,\sfd,\mm)$ a ${\sf PI}$ space, there exists a Borel partition $\{U_i\}_i$ of $\X$ (up to $\mm$-negligible set) and Lipschitz $\varphi_i \colon \X \to \mathbb{R}^{N_i}$ such that for every Lipschitz function $f \colon \X \to \mathbb{R}$, there exists a function $\d f\colon U_i \to {\mathbb{R}^{N_i}}^*$ such that
        \begin{equation*}
            \limsup_{y \to x}\frac{|f(y)-f(x)-\d f_x(\varphi_i(y)-\varphi_i(x))|}{\sfd(x,y)} =0,\quad \text{for }\mm\text{-a.e.\ }x \in U_i.
        \end{equation*}
    \item \emph{Regularity of harmonic functions}.
    The possibility of defining a differential $\d f_x$ of a Lipschitz function $f$ at a point $x$ in a ${\sf PI}$ space allows to speak about harmonic functions in the sense of distributions. These functions are usually denoted as Cheeger-harmonic functions and satisfy a Caccioppoli-type inequality. The coupling of the Poincar\'{e}-inequality together with a Caccioppoli-inequality is a key ingredient for Koskela-Rajala-Shanmugalingam in \cite{KoskelaRajalaShanmugalingam2003} in proving Lipschitz regularity of Cheeger-harmonic functions. Other additional regularity assumptions on the space are needed, but we do not want to stress this point further.
\end{itemize}

Although the definition of ${\sf PI}$ spaces is analytical, it can be characterized in purely geometric terms. This shows a nice interplay between techniques in analysis and geometry when studying ${\sf PI}$ spaces. Moreover, it may be difficult to verify that a given space is a ${\sf PI}$ space by directly applying the definition. Therefore, certain geometric characterizations are more suitable for the geometry of certain class of spaces and they are a more handy tool to prove they are ${\sf PI}$ spaces.

A recurrent rule of thumb is that this class of spaces satisfies two geometric conditions, that we now formulate in very informal terms:
\begin{itemize}
    \item \emph{Quantitative connectivity (dimension 1)}. Every pair of points can be connected by `many' rectifiable paths.
    \item \emph{Absence of bottlenecks (codimension 1)}. Every pair of points can be separated by a sufficiently large hypersurface.
\end{itemize}

The main goal of this note is to make these intuitions precise and to explain some characterizations related to them. We can formulate both conditions in a way that they characterize when a doubling metric measure space is a ${\sf PI}$ space. 

The first property can be precisely formulated either in terms of bounds on the modulus of paths connecting points or in terms of a probability measures on paths connecting points with certain properties. In this context, we also recall the obstacle-avoidance principle.  This will be the content of Section \ref{sec:dimension_1}.

The second property can be formulated either in terms of relative isoperimetric inequality or in terms of the surface area of separating sets, weighted with a Riesz kernel. This will be the content of Section \ref{sec:codimension_1}.

Before going into the aforementioned characterizations, one may wonder how large the class of ${\sf PI}$ spaces is. Relevant classes of examples are:
\begin{itemize}
    \item[1)] \emph{Smooth connected complete closed Riemannian manifolds} with Ricci curvature greater or equal than $0$ \cite{Buser1982}.
    \item[2)] \emph{Metric measure spaces satisfying the curvature-dimension condition}. This class generalizes 1), was defined and studied in \cite{SturmI,SturmII,LottVillani09}. It includes Ricci limit spaces and is compact and stable with respect to (pointed) measured Gromov-Hausdorff convergence. ${\sf CD}(0,N)$ spaces with $N \in (1,\infty)$ are ${\sf PI}$ spaces, as proved in \cite{Rajala12}.
    \item[3)] \emph{Sub-Riemannian geometries}. These are smooth manifolds endowed with a Riemannian metric which takes infinite value on some directions on the tangent spaces. 
    Examples of ${\sf PI}$ spaces in this class are the Heisenberg group and, more generally, any Carnot group equipped with the Lebesgue measure and the Carnot–Carathéodory metric (see \cite[Proposition 11.17]{HK00}, after \cite{Varopoulus87}). 
    \item[4)] \emph{Non self-similar Sierpi\'{n}ski carpet} \cite{MackayTysonWildrick2013}. The non self-similar Sierpi\'{n}ski carpet is a modification of the classical construction of the Sierpi\'{n}ski carpet in the Euclidean plane. These are constructions that can be thought as a fattening of the Sierpi\'{n}ski carpet by removing smaller and smaller squares at each stage (see \cite{Sylvester-Gong-21} for an alternative proof and for a generalization of the result to higher dimensions).
    \item[5)] \emph{Laakso spaces}  \cite{Laakso2000}. These are a family of metric measure spaces parametrized by $1 \le Q\in \mathbb{R}$ that have Hausdorff dimension $Q$ and do not biLipschitz embed in any Euclidean space. 
    See also \cite[Theorem 2.3]{LangPlaut2001} for an alternative construction.
\end{itemize}

We recall that the doubling condition can be generalized to a local doubling condition and that the Poincar\'{e} inequality can be formulated with a $p$-norm of $\lip u$ on the right-hand side with $p > 1$. This type of generalization allows for more flexibility in the previous class of examples, but we would like to keep the presentation simpler and focus on our axiomatization.

Other relevant contributions related to examples and characterization of  the $p$-Poincar\'{e} inequality for $p > 1$ can be found in \cite{ErikssonBique2019, Sylvester-Gong-21, KorteLahti2014, MackayTysonWildrick2013}.

\subsection{Structure of the overview}

Section \ref{sec:preliminaries} contains some preliminaries about analysis on metric spaces, Lipschitz functions, Riesz kernels and Poincar\'{e} inequality.

Section \ref{sec:dimension_1} surveys characterizations of ${\sf PI}$ spaces in terms of curves, as Keith's modulus estimates, pencil of curves and the obstacle-avoidance principle. We sketch the proofs of the existence of pencil of curves. Section \ref{sec:codimension_1} presents the characterizations in terms of objects of codimension 1. Here, we review isoperimetric inequalities and our contributions in \cite{CaputoCavallucci2024II} in terms of energy of separating sets. Here, we show how separating sets can be used to prove that a space is ${\sf PI}$, by tailoring the discussion to $\mathbb{R}^d$.

Section \ref{sec:a_more_geometric_characterization} discusses a direct proof of the equivalence of ${\sf PI}$ spaces, obstacle-avoidance principle with Riesz kernel and Minkowski content of separating sets, as proved in \cite{CaputoCavallucci2024II}.

We then review the main result of \cite{CaputoCavallucci2024II}, where a direct and more geometric relation between the obstacle-avoidance principle and separating sets can be achieved, without appealing to the Poincar\'{e} inequality. The argument is geometric and uses the so-called position function. We introduce this function in Section \ref{sec:position_function} and explain its properties and its relation to the above-mentioned equivalence. In Section \ref{sec:open}, we present some open problems.

\section*{Acknowledgements}

I am supported by the European Union’s Horizon
2020 research and innovation programme (Grant agreement No. 948021). I would like to thank Nicola Cavallucci, Francesco Nobili, Tapio Rajala and Pietro Wald for reading a preliminary version of the manuscript. I also thank Fabio Cavalletti, Matthias Erbar, Jan Maas and Karl-Theodor Sturm for the organization of the workshop `School and Conference on
Metric Measure Spaces, Ricci Curvature, and Optimal Transport' at
Villa Monastero, Lake Como and for the opportunity
to write a conference proceeding of the event.

\section{Preliminaries}

\label{sec:preliminaries}
Given a metric space $(\X,\sfd)$, we denote by $\mathscr{B}(\X) \subset 2^\X$ the family of Borel subsets of $\X$. 
We denote by $\mathcal{P}(\X)$ the convex set of nonnegative Borel probability measures on $(\X,\sfd)$.
We define $B_r(x):=\{y \in \X\,:\,\sfd(x,y) < r\}$ and $\overline{B}_r(x):=\{y \in \X\,:\,\sfd(x,y) \le r\}$.
Given a metric space $(\X,\sfd)$ and $k \in \mathbb{N}$, we denote by $\mathcal{H}^k$ the $k$-dimensional Hausdorff measure. We denote by $\mathcal{L}^d$ the Lebesgue measure in $\mathbb{R}^d$.

\subsection*{Curves}
We denote by $C([0,1],\X)$ the space of continuous curves from $[0,1]$ with values in $\X$. We endow it with the distance $\sfd_{\infty}(\gamma,\eta):={\rm sup}_{t \in [0,1]}\sfd(\gamma(t),\eta(t))$.

We define the length of a curve $\gamma \in C([0,1],\X)$ as
\begin{equation*}
    \ell(\gamma):= \sup \left\{ \sum_{i=0}^{N-1} \sfd(\gamma(t_i), \gamma(t_{i+1})) \,:\, N \in \mathbb{N},\,
        0=t_0 < t_1 <\dots < t_N =1\right\}.
\end{equation*}

A curve $\gamma \in C([0,1],\X)$ is said to be rectifiable if $\ell(\gamma)<\infty$. Every rectifiable curve admits a Lipschitz reparametrization, i.e.\ there exists an increasing map $\phi \colon [0,1] \to [0,1]$ with $\phi(0)=0$ and $\phi(1)=1$ such that $\tilde{\gamma}:=\gamma \circ \phi$ is $\ell(\gamma)$-Lipschitz. We say that $\gamma \in C([0,1],\X)$ is a geodesic if $\ell(\gamma)=\sfd(\gamma(0),\gamma(1))$.

Given $L \ge 1$, we say that $\gamma \in C([0,1],\X)$ is an $L$-quasigeodesic if $\ell(\gamma) \le L \sfd(\gamma(0),\gamma(1))$. We say that it is an $L$-quasigeodesic connecting $x,y \in \X$ if it is an $L$-quasigeodesic and $\gamma(0)=x$ and $\gamma(1)=y$. We denote the set of $L$-quasigeodesics connecting $x$ to $y$ with $\Gamma_{x,y}^L$. The notion of quasigeodesic quantifies how far a curve is from being a geodesic (notice that it always holds that $\sfd(\gamma(0),\gamma(1))\le \ell(\gamma)$ for $\gamma \in C([0,1],\X)$ by the very definition of $\ell(\gamma)$). We say that a metric space is $L$-quasiconvex if for every couple of points $x,y\in \X$ there exists an $L$-quasigeodesic connecting $x$ to $y$. Similarly, we define the set of rectifiable curves connecting $x$ to $y$ and we denote it by $\Gamma_{x,y}$.

For every Lipschitz curve $\gamma\colon [0,1] \to \X$ the function $|\dot{\gamma}(\cdot)|\colon [0,1] \to [0,\infty)$ defined as $|\dot{\gamma}(t)|:=\lim_{h \to 0}\frac{\sfd(\gamma(t+h),\gamma(t))}{|h|} \in [0,\infty)$ is almost everywhere well-defined and it is called the metric speed of $\gamma$ (\cite[Theorem 1.1.2]{AmbrosioGigliSavare}).

Given a Borel function $g \colon \X \to [0,\infty)$ and $\gamma \colon [0,1]\to \X$ Lipschitz we define
\begin{equation*}
    \int_\gamma g\,\d s := \int_0^1 g(\gamma(t))\,|\dot{\gamma}(t)|\,\d t.
\end{equation*}
This quantity does not depend on the Lipschitz reparametrizations of $\gamma$.
With some abuse of notation, we write $\int_\gamma g\,\d s$ when $\gamma$ is rectifiable and we mean $\int_\gamma g\,\d s:=\int_{\tilde{\gamma}} g\,\d s$, where $\tilde{\gamma}$ is a Lipschitz reparametrization of $\gamma$.
Given a Borel set $A\subset \X$ and a rectifiable curve $\gamma \in C([0,1],\X)$, we define
\begin{equation*}
    \ell(\gamma\cap A):=\int_{\gamma} \chi_A \,\d s.  
\end{equation*}

\subsection*{Lipschitz functions}
\label{sec:lipschitz_functions}

The \emph{local Lipschitz constant} $\lip u \colon \X \to \mathbb{R}$ of a function $u\colon \X \to \R$ is
$$\lip u (x):=\lims_{y \to x} \frac{|u(y)-u(x)|}{\sfd(y,x)} = \lim_{\delta \to 0} \sup_{y\in B_\delta(x)\setminus \{x\}} \frac{|u(y)-u(x)|}{\sfd(y,x)} $$ 
with the convention that $\lip u (x)=0$ if $x$ is an isolated point. If $u$ is locally Lipschitz, we have that for every rectifiable $\gamma$
\begin{equation}
\label{eq:lipu_is_an_uppergradient}
    |u(\gamma(1))-u(\gamma(0))| \le \int_\gamma \lip u \,\d s.
\end{equation}
The above implies that $\lip u$ is an upper gradient of $u$ (see \cite[Lemma 6.2.6]{HKST15}), where an upper gradient is defined as follows. A Borel function $g\in \X \to [0,\infty]$ is an upper gradient of a function $u \colon \X \to \mathbb{R}$ if
$$|u(\gamma(1))-u(\gamma(0))|\le \int_\gamma g \,\d s$$
for every rectifiable curve $\gamma$.

\subsection*{Riesz kernel}
\label{sec:Riesz_kernel}

We define one of the central objects of this survey: the \emph{Riesz potential} or \emph{Riesz kernel}. Given $x,y\in \X$, the \emph{Riesz potential} $R_{x,y} \colon \X \to [0,\infty)$ with poles at $x$ and $y$ is defined for every $z \in \X\setminus \lbrace x,y \rbrace$ as
\begin{equation}
\begin{aligned}
    R_{x,y}(z):&= \frac{\sfd(x,z)}{\mm(B_{\sfd(x,z)}(x))} + \frac{\sfd(y,z)}{\mm(B_{\sfd(y,z)}(y))}=: R_x(z)+R_y(z).
\end{aligned}
\end{equation}
Moreover, we define $R_{x,y}(x) = R_{x,y}(y) = 0$. For $L\geq 1$ we set $B_{x,y}^L := B_{2L\sfd(x,y)}(x) $ and $\overline{B}_{x,y}^L := \overline{B}_{2L\sfd(x,y)}(x) $. The $L$-\emph{truncated Riesz potential with poles at $x,y$} is
\begin{equation}
    R_{x,y}^L(z):= \chi_{B_{x,y}^L}(z) R_{x,y}(z)
\end{equation}
for every $z \in \X\setminus \lbrace x,y \rbrace$. The corresponding Riesz measure is defined as
\begin{equation}
    \mm_{x,y}^L = R^L_{x,y}\,\mm. 
\end{equation}
It is a measure on $\X$ which is supported on $\overline{B}_{x,y}^L$. This measure has been already studied for instance in \cite{Hei01} and \cite{Kei03}. 
We list some additional properties of the Riesz kernel and of the measure $\mm_{x,y}^L$, proven in \cite[Proposition 2.3 and Lemma 2.4]{CaputoCavallucci2024}.

Given a $C_D$-doubling metric measure space $(\X,\sfd,\mm)$ and $x,y\in \X$ and $L\geq 1$. We have that $\mm_{x,y}^L(\X) \leq 8C_DL\sfd(x,y)$. In particular $\mm_{x,y}^L$ is a finite Borel measure.

Moreover, if $(\X,\sfd)$ is geodesic, then the map
    $$\mathcal{D}(R) \to [0,+\infty], \quad (x,y,z) \mapsto R_{x,y}(z)$$
    is continuous, where $\mathcal{D}(R) := \lbrace (x,y,z)\in \X : x,y,z\text{ distinct}\rbrace$.

\subsection{Pointwise estimates}
\label{sec:pointwise_estimates}

We can now state two equivalent formulations of the Poincaré inequality on doubling metric measure spaces, as in \cite{Hei01}. We refer the reader to \cite[Theorem A.3]{CaputoCavallucci2024}.
\begin{proposition}[{\cite[Theorem 9.5]{Hei01}}]
\label{prop:Poincaré_equivalences_pointwise}
    Let $(\X,\sfd,\mm)$ be a doubling metric measure space. The following are quantitatively equivalent:
    \begin{itemize}
        \item[(PI)] $(\X,\sfd,\mm)$ is a ${\sf PI}$ space;
        \item[(PtPI)] $\exists C > 0$, $L\geq 1$ such that for every $x,y \in \X$ and every $u\in \Lip(\X)$ it holds
        \begin{equation}
        \label{eq:Riesz_PtPI}
            |u(x)-u(y)|\le C\, \int_\X \lip u\,\d \mm_{x,y}^{L}.
        \end{equation} 
    \end{itemize}
    The same inequality equivalently holds for the couple $(u,g)$ replacing $(u, \lip u)$, where $g$ is an upper gradient of $u$.
\end{proposition}

\section{Characterizations in terms of objects of dimension 1}
\label{sec:dimension_1}

As explained in the introduction, the Poincar\'{e} inequality is related to the fact that every couple of points can be connected by many curves. 

A first necessary condition is given by the following result. The argument is due to Semmes, see the proof in \cite[Appendix]{Cheeger99} (or the textbook \cite[Theorem 8.3.2]{HKST15}).

\begin{proposition}
Let $(\X,\sfd,\mm)$ be a ${\sf PI}$ space. Then there exists a constant $L \ge 1$ such that $(\X,\sfd)$ is $L$-quasiconvex.    
\end{proposition}

First of all, this condition gives some topological constraint on the possibility of endowing a metric space with a Borel measure that makes it a ${\sf PI}$ space. It forces the topological space to be path-connected.
Moreover, it is odd to expect that such a necessary condition characterizes a ${\sf PI}$ space, as quasiconvexity is a purely metric condition and does not involve the measure $\mm$.

\subsection{Pencil of curves}

We refer the reader to Section \ref{sec:Riesz_kernel} for the preliminaries about the Riesz kernel. We define the core notion of this section and state the main theorem.

\begin{definition}[Pencil of curves \cite{Hei01,Semmes}]
      We say that a metric measure space $(\X,\sfd,\mm)$ admits a pencil of curves if there exist two constant $C_1>0$,$L\ge 1$ such that for every couple of points $x,y\in \X$ there exists $\alpha \in \mathcal{P}(\Gamma_{x,y}^L)$ such that 
      \begin{equation}
      \label{eq:pencil}
          \int \int_{\gamma} g\,\d s\,\d \alpha(\gamma) \le C_1 \int g R_{x,y}^L\,\d \mm
      \end{equation}
      for every nonnegative Borel function $g$.
\end{definition}

\begin{theorem}[Characterization 1, \cite{DurCarErikBiqueKorteShanmu21,FasslerOrponen19}]
\label{thm:characterization_1}
    A doubling metric measure space $(\X,\sfd,\mm)$ is a ${\sf PI}$ space if and only if it admits a pencil of curves.
\end{theorem}

The if part has been proved by Semmes in the 1990s and directly follows by the definition of upper gradient. The only if implication is more recent and there are two proofs of this implication, related to the two independent contributions \cite{DurCarErikBiqueKorteShanmu21,FasslerOrponen19}.
\begin{itemize}
    \item[1)] \emph{Sketch of the proof of \cite{DurCarErikBiqueKorteShanmu21}}.
    The proof is based on the combination of the pointwise estimates of Section \ref{sec:pointwise_estimates} and the following abstract min-max theorem in convex optimization ({\cite[Thm.\ 9.4.2]{Rudin80} (original proof in \cite{Sion58})}).

    \begin{theorem}
    \label{thm:minmax}
    Let $\X_1$ be a vector space and $\X_2$ be a topological vector space. Let $G \subseteq \X_1$ and $K \subseteq \X_2$ be convex subsets, with $K$ compact. Let $F \colon G \times K \to \mathbb{R}$ be such that
\begin{itemize}
    \item[a)] $F(\cdot,y)$ is convex on $G$ for every $y \in K$;
    \item[b)] $F(x,\cdot)$ is concave and upper semicontinuous in $K$ for every $x \in G$.
\end{itemize}
Then
    \begin{equation}
        \max_{y \in K} \inf_{x \in G} F(x,y) =  \inf_{x \in G} \max_{y \in K} F(x,y).
    \end{equation}
    \end{theorem}
    The proof goes as follows. (PI) implies (PtPI) by Proposition \ref{prop:Poincaré_equivalences_pointwise}.
    The (PtPI) condition implies that there exists a constant $C >0$ such that for every $x,y \in \X$ and every nonnegative continuous function $f$, there exists a rectifiable curve connecting $x$ to $y$ with
    \begin{equation}
    \label{eq:A_p_connectdness_item1}
        \int_\gamma f\,\d s \le C \int f \,\d \mm_{x,y}^L.
    \end{equation}
    This is a consequence of a, by now classical, argument in metric analysis that goes as follows. We fix a distinguished point $x_0$ and we define $u \colon \X \to [0,\infty]$ as
    \begin{equation*}
        u(x):=\inf\left\{ \int_{\gamma} f\,\d s\,:\, \gamma \text{ is rectifiable, }\gamma(0)=x,\, \gamma(1)=x_0\,\right\}.
    \end{equation*}
    Then $f$ is an upper gradient of $u$ and we apply (PtPI) to the couple $(u,f)$ and we get \eqref{eq:A_p_connectdness_item1}. Now, a key point is to embed the space of rectifiable curves into $\mathcal{P}(\Gamma_{x,y})$ via the map $\gamma \mapsto \delta_\gamma$. This gives, after extending \eqref{eq:A_p_connectdness_item1} to all nonnegative Borel functions
    \begin{equation*}
        \sup_{f \text{ continuous }}\,\inf_{\alpha \in \mathcal{P}(\Gamma_{x,y})}\,\left(C \int f \,\d \mm_{x,y}^L-\int \int_\gamma f\,\d s\,\d \alpha(\gamma)\right)>0.  
    \end{equation*}

    We set $G:=\left\{f\colon \X \to \mathbb{R}:\,f\text{ is continuous.} \right\}$, $K:=\mathcal{P}(\Gamma_{x,y}^L)$, that is convex and compact with respect to the weak$^*$ topology and $F(\alpha,f):=C \int f \,\d \mm_{x,y}^L-\int \int_\gamma f\,\d s\,\d \alpha(\gamma)$. We apply Theorem \ref{thm:minmax} with these choices and we obtain the conclusion.
    
    \item[2)] \emph{Sketch of the proof of \cite{FasslerOrponen19}}. We fix a couple of points $x,y \in \X$ with $x \neq y$ and we consider a $\delta$-net, i.e.\ a maximal set of points $\{x_i\}_{i\in I}$, where $I\subset \mathbb{N}$ and such that $\sfd(x_i,x_j) \ge \delta$ for $i \neq j$. If $\delta<\sfd(x,y)$, we can assume that $x,y$ belong to such $\delta$-net.
    W define a graph $(V,E)$ with $E \subset V \times V$, where $V=\{x_i\}_{i\in I}$, by declaring that $(x_i,x_j) \in E$ if and only if $i\neq j$ and $B_{2\delta}(x_i) \cap B_{2\delta}(x_j) \neq \emptyset$.

    \begin{figure}[h]
        \centering
        \includegraphics[scale=0.7]{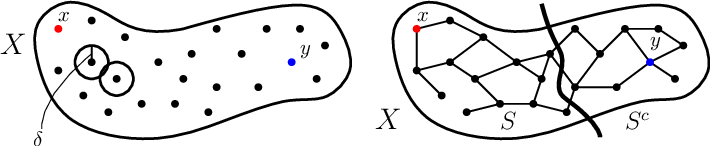}
        \caption{Example of the construction of the graph and of a cut $S$.}
    \end{figure}
    
    We associate a capacity $c \colon E \to (0,\infty)$, defined as
    \begin{equation*}
        c(x_i,x_j):=\frac{\mm(B_\delta(x_i))}{\delta} \frac{\sfd(x,x_i)}{\mm(B_{\sfd(x,x_i)}(x))}+ \frac{\mm(B_\delta(x_j))}{\delta} \frac{\sfd(y,x_j)}{\mm(B_{\sfd(y,x_j)}(y))}.
    \end{equation*}

    For every cut $S$, that is a set $S \subset V$ such that $x \in S$ and $y \notin S$, define
    \begin{equation*}
        \mathcal{C}(S):=\sum_{(z,w) \in E, z \in S,\,w \in S^c} c(z,w).
    \end{equation*}
    
    Using the PI assumption, we have that $\inf\left\{  \mathcal{C}(S)\,:\,S\text{ is a cut}\right\}=c_0>0$ (with the bound $c_0$ independent of $\delta$).
    \emph{Ford-Fulkerson algorithm} (\cite{FordFulkerson}) gives a flow on the graph $f \colon E \to \R$ such that
    \begin{equation*}
        \sum_{(x,z) \in E} f(x,z) = \inf\left\{  \mathcal{C}(S)\,:\,S\text{ is a cut} \right\}.
    \end{equation*}
    We do not further specify the notion of flow. It can be thought as an association of numbers to every edge that is less or equal than the capacity and satisfies some balance condition at every vertex.
    
    We take the limit in the sequence of flows in the following sense. Every flow induces a metric current in the sense of Ambrosio-Kircheim \cite{AK00} and the limit is intended with respect to weak convergence in the space of currents. The limit current $T$ is not trivial because of ${\sf PI}$ assumption and is normal (indeed $\partial T= \delta_y-\delta_x$). By \cite{PaoliniStepanov}, it induces a measure on curves with endpoints $x,y\in \X$ that is a $1$-pencil.
    
\end{itemize}

The proof in item 2) shows that `having large cuts' in a graph is equivalent to being a ${\sf PI}$ space. This can be seen as a discrete and combinatorial counterpart of the result we obtained in \cite{CaputoCavallucci2024}. We point out that both proofs shares the presence of a duality-type argument, i.e.\ the Fold-Fulkerson algorithm or the Sion min-max principle.

\subsection{Modulus estimates}

Another quantification on the family of curves is given by the modulus of a family of curves.

We define the class of admissible functions for a family $\Gamma$ of rectifiable curves as
\begin{equation*}
    {\rm Adm}(\Gamma):=\left\{ \rho \colon \X \to [0,\infty]\,:\,\rho \text{ is Borel and }\int_\gamma \rho \ge 1 \text{ for every }\gamma \in \Gamma \right\}.
\end{equation*}

We define the modulus of a family of rectifiable curves $\Gamma$ with respect to a nonnegative Borel measure $\mu$ as 
\begin{equation*}
    {\rm Mod}(\Gamma, \mu):=\inf \left\{ \int_\X \rho \,\d \mu\,:\, \rho \in {\rm Adm}(\Gamma) \right\}.
\end{equation*}
It follows by the very definition that, if $\Gamma_1,\Gamma_2$ are families of rectifiable curves with $\Gamma_1 \subset \Gamma_2$, then ${\rm Mod}(\Gamma_1,\mu) \le {\rm Mod}(\Gamma_2,\mu)$. Thus, enlarging a set of curves will increase its modulus.
One can think that the modulus of a family of curves quantifies how large a family of curve is in a measure-theoretic sense. This metric measure quantification of curves allow to characterize the Poincar\'{e} inequality. 
\begin{theorem}[Characterization 2, \cite{Kei03}]
\label{thm:characterization_2}
A doubling metric measure space is a ${\sf PI}$ space if and only if there exist constants $C_2>0,L \ge 1$ such that for every couple of points $x,y \in \X$ we have 
\begin{equation}
\label{eq:modulus_lower_bound}
    {\rm Mod}(\Gamma^L_{x,y},\mm_{x,y}^L) \ge C_2.
\end{equation}
\end{theorem}

We show how Theorem \ref{thm:characterization_2} can be proved as a corollary of Theorem \ref{thm:characterization_1}.

\begin{proof}
We prove the only if implication. Since $(\X,\sfd,\mm)$ is a ${\sf PI}$ space we have that for every couple of points $x,y\in \X$ there exists $\alpha \in \mathcal{P}(\Gamma_{x,y}^L)$ such that \eqref{eq:pencil} holds.
Let $\rho \in {\rm Adm}(\Gamma_{x,y}^L)$. In particular,
\begin{equation*}
    \int \int_\gamma \rho \,\d s \, \d \alpha(\gamma) \ge \alpha(\Gamma_{x,y}^L)=1.
\end{equation*}
This last inequality, in combination with \eqref{eq:pencil} gives that $\int \rho \,\d \mm_{x,y}^L \ge C_1^{-1}$. Therefore, this implies \eqref{eq:modulus_lower_bound} with $C_2:=C_1^{-1}$.
For the if implication, we argue as in \cite{Kei03}. Fix two points $x,y \in \X$. Given a locally Lipschitz function $u$, we consider $g:=\lip u/|u(x)-u(y)|$. By \eqref{eq:lipu_is_an_uppergradient} $g\in {\rm Adm}(\Gamma_{x,y}^L)$, therefore by assumption
\begin{equation*}
    \frac{ \int \lip u\,\d \mm_{x,y}^L}{|u(x)-u(y)|} \ge \int g \,\d \mm_{x,y}^L \ge {\rm Mod}(\Gamma_{x,y}^L,\mm_{x,y}^L) \ge C_2.
\end{equation*}
By repeating the argument for every $x,y \in \X$, from Proposition \ref{prop:Poincaré_equivalences_pointwise} we get the conclusion.
\end{proof}

\subsection{Obstacle-avoidance principle}
\label{sec:obstacle_avoidance}

The obstacle-avoidance principle says that for a given set and a given couple of points, one can always find a rectifiable path joining the points that spends little time in the set.
This amount of time is considered small when compared to the size of the set, that can be computed with two different choices of energy.

A first example of energy is given by the \emph{Hardy-Littlewood maximal function}. We define

\begin{equation*}
    M_s f(x):=\sup_{0<r<s} \dashint_{B_r(x)} |f|\,\d \mm\quad\text{for }f \in L^1_{{\rm loc}}(\X).
\end{equation*}

Eriksson-Bique in \cite{ErikssonBique2019II} defined the following condition (a similar condition was previously discussed in \cite{ErikssonBique2019}).

\begin{definition}[$A_1$-connectdness \cite{ErikssonBique2019II}]

    There exist two constants $L\ge 1,C_A> 0$ such that for every nonnegative lower semicontinuous function $g$ and for every $x,y$, there exists a Lipschitz curve $\gamma \in \Gamma_{x,y}^L$ such that 
    \begin{equation}
    \label{eq:Ap_connectdness}
        \int_\gamma g\,\d s\le C_A \sfd(x,y) \left( M_{C \sfd(x,y)} g (x) + M_{C \sfd(x,y)} g (y) \right).
    \end{equation}
\end{definition}

Given a doubling metric measure space, the condition in the previous definition can be thought as a quantification of quasiconvexity and characterizes ${\sf PI}$ spaces. Thus, to prove the Poincar\'{e} inequality, it suffices to find for every function $g$ as in the definition a curve $\gamma$ with controlled length that `travels' in the region where $g$ is sufficiently small, in the sense of \eqref{eq:Ap_connectdness}.

The subscript $1$ is related to the fact that an $A_p$ condition can be formulated and characterized the $p$-Poincar\'{e} inequality for $p \ge 1$. The self-improvement of the $A_p$ condition was used in \cite{ErikssonBique2019II} to reprove with a short argument Keith-Zhong's result about self-improvement of the $p$-Poincar\'{e} inequality.

It turns out that working with sets instead of functions allows to use a smaller class of objects in order to check that a space is a ${\sf PI}$ space.

\begin{definition}[Maximal connectivity {\cite[Definition 2.12]{Sylvester-Gong-21}}]
Let $L\ge 1$ and $C >0$. We say that a metric measure space $(\X,\sfd,\mm)$ is $(L,C)$-max connected if for every couple of points $x,y \in \X$ and for every Borel set $E\subset \X$ there exists a Lipschitz curve $\gamma \in \Gamma_{x,y}^L$ such that
\begin{equation*}
    \ell(\gamma \cap E) \le C \sfd(x,y) \left(M_{Cr}(\chi_E)(x)+ M_{Cr}(\chi_E)(y) \right).
\end{equation*}
\end{definition}

As before, also in this case, the authors present natural variants of the definition for $p >1$ that are related a to a $p$-Poincar\'{e} inequalities. This condition was used in \cite{Sylvester-Gong-21} in different contexts. For instance, they characterize the Sierpi\'{n}ski sponges that satisfy a $p$-Poincar\'{e} inequality and the range of $p$ for which the inequality holds (generalizing a result in \cite{MackayTysonWildrick2013} in dimension 2) and they give new examples of Sobolev extension domains in the metric setting.
The relation with ${\sf PI}$ spaces (or more generally, to the validity of $p$-Poincar\'{e} inequalities) is given by the following statement (a combination of \cite[Theorems 2.18, 2.19 and Lemma 2.20]{Sylvester-Gong-21}).

\begin{proposition}
Let $(\X,\sfd,\mm)$ be a doubling metric measure space. If $\X$ is a ${\sf PI}$ space, then then it is $(L,C)$-max connected for some constants $L \ge 1$ and $C>0$. Conversely, if $\X$ is $(C,L)$-max connected for some constants $C>0$ and $L \ge 1$, then it satisfies a $p$-Poincar\'{e} inequality for every $p >1$.
\end{proposition}

The obstacle-avoidance principle can be alternatively formulated using the Riesz kernel in place of the maximal function. This allows a characterization of ${\sf PI}$ spaces and it is one of the main novelties introduced in \cite{CaputoCavallucci2024II}.
To this aim, we introduce the following auxiliary quantity.
We define the width of a set $A \subseteq \X$ with respect to the points $x,y\in \X$ as 
$$\width_{x,y}(A) := \inf_{\gamma \in \Gamma_{x,y}} \ell(\gamma \cap A),$$
where $\Gamma_{x,y}$ is the set of rectifiable paths connecting $x$ to $y$. The quantity $\width_{x,y}(A)$ measures the width of the set $A$ in the following sense: we consider all the curves (with finite length) connecting $x$ to $y$ and we look at the one whose length inside $A$ is minimal.

\begin{definition}[1-set-connectedness]
\label{def:CLLsetconnected}
    Let $C>0$ and $L \geq 1$. We say that $(\X,\sfd,\mm)$ is $(C,L)$ $1$-set-connected at $x,y\in \X$ if
\begin{equation}
    \label{eq:defin_set_connectedness_bounded_intro}
    \width_{x,y}(A)\leq C\mm_{x,y}^L(A) \text{ for all } A\subseteq \X  \text{ Borel}.
\end{equation}
\end{definition}

There are several equivalent definitions to 1-set-connectedness. The condition can be asked for all closed set instead of all Borel sets and the width can be defined as the minimum over $\tilde{L}$-quasigeodesics connecting $x$ to $y$ for some $\tilde{L} \ge 1$.

A positive side of this definition is that it characterizes ${\sf PI}$ spaces. We will see the proof in Section \ref{sec:a_more_geometric_characterization}.

\begin{proposition}[Characterization 3 {\cite[Theorem 1.4]{CaputoCavallucci2024II}}]
\label{prop:characterization_3}
    Let $(\X,\sfd,\mm)$ be a doubling metric measure space. Then it is a $\sf PI$ space if and only if $\X$ is $(C,L)$ 1-set-connected for some constants $C>0,L \ge 1$.
\end{proposition}

Actually, in \cite[Theorem 1.4]{CaputoCavallucci2024II} we prove a stronger statement. We prove that the condition of 1-set-connectedness verified only for a fixed couple of points $x,y\in\X$ is equivalent to the inequality \eqref{eq:Riesz_PtPI} at the same fixed couple of points for all Lipschitz functions $u$. Then, Proposition \ref{prop:characterization_3} follows as a consequence of Proposition \ref{prop:Poincaré_equivalences_pointwise}.

\begin{example}
    We study this characterization in the simple case of the Euclidean space $\mathbb{R}^d$ and we show the scaling of the quantities involved in the definition of $(C,L)$ 1-set-connectedness. We fix two points $x,y \in \mathbb{R}^d$ and we look at $A=B_r(x)$ with $0<r<\sfd(x,y)$.  
    \begin{figure}[h!]
        \centering
        \includegraphics[width=0.4\linewidth]{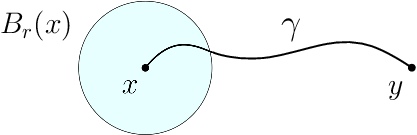}
        \caption{Computation in the toy-example of $\X=\mathbb{R}^d$ and $A=B_r(x)$.}
        \label{fig:obstacle_avoidance}
    \end{figure}
    By the definition of $\width_{x,y}(B_r(x))$, we have that the infimum in the definition is realized by every rectifiable curve connecting $x$ to $y$ that `travels once' in a radial direction inside $B_r(x)$.
    Thus, $\width_{x,y}(B_r(x))=r$. For what concerns the right-hand side, we compute with $L=1$
    \begin{equation*}
    \begin{aligned}
        \mm^L_{x,y}&(B_r(x))=\int_{B_r(x)} R_x\,\d\mathcal{L}^d+ \int_{B_r(x)} R_y\,\d\mathcal{L}^d\\
        &= \int_0^r {\omega_d}^{-1} s^{1-d} \mathcal{H}^{d-1}(\partial B_s(x))\,\d s + \int_{B_r(x)} R_y\,\d\mathcal{L}^d=\frac{\omega_{d-1}}{\omega_d} r +\int_{B_r(x)}R_y \d \mathcal{L}^d \ge \frac{\omega_{d-1}}{\omega_d} r.
    \end{aligned}
    \end{equation*}
    Notice that, for $r$ approaching $0$, $\mm_{x,y}^L(B_r(x))=\omega_{d-1}\omega_d^{-1} r+O(r^d)$. The computation shows that $(\omega_{d}/\omega_{d-1},1)$ 1-set-connectedness is verified for such balls. The general case is more complicated. Indeed, one needs to study
    \begin{equation*}
        \inf_{A \subset \X:\, A \text{ Borel}} \frac{\mm_{x,y}^L(A)}{\width_{x,y}(A)}
    \end{equation*}
    suitably defining the ratio in the undetermined cases. So it suffices to study the minimizers of such ratio and compute the ratio only for such sets. We do this by connecting this theory with the theory of separating sets in \cite{CaputoCavallucci2024II}. This is the content of Section \ref{sec:a_more_geometric_characterization}.
\end{example}

\section{Characterizations in terms of objects of codimension 1}
\label{sec:codimension_1}

The characterizations of the next sections involve boundaries of sets and the notion of surface area associated to them. There are several notions of what the boundary of a Borel set $E$ is: either the topological boundary $\partial E$ or the \emph{measure-theoretic boundary} 
\begin{equation*}
    \partial^e E = \left\{ x\in \X : \lims_{r\to 0}\frac{\mm(B_r(x)\cap E)}{\mm(B_r(x))} > 0 \text{ and } \lims_{r\to 0}\frac{\mm(B_r(x)\setminus E)}{\mm(B_r(x))} > 0  \right\}.
\end{equation*}
It holds that $\partial^e E \subset \partial E$.

There are several notions of surface area. The
{\em total variation} of $u \in L^1_\loc (X)$ on an open set $A \subset X$ is
\[ 
|D u|(A) := \inf \left\{ \liminf_{i\to \infty} \int_A \lip u_i \,\d\mm :\, u_i \in \Lip_\loc(A),\, u_i \to u \mbox{ in } L^1_\loc(A) \right\}. 
\]
For an arbitrary set $B\subset X$, we define
\[
|D u|(B):=\inf\{|D u|(A):\, B\subset A,\,A\subset X
\text{ open}\}.
\]
The set function $|D u|\colon \mathscr{B}(\X) \to [0,\infty]$ is a
Borel measure on $\X$ by \cite[Theorem 3.4]{Mir03} (see also \cite[Lemma 5.2]{AmbrosioDiMarino14}).
We say that a Borel set $E \subseteq \X$ has finite perimeter if $|D\chi_E|(\X) < \infty$. In this case, we set ${\rm Per}(E,\cdot):=|D \chi_E|$.

Next, we introduce the codimension-$1$ Hausdorff measure. 
We denote by $\codH{1}_\delta$ the pre-measure with parameter $\delta >0$ defined as
\begin{equation}
    \codH{1}_\delta(E):=\inf \left\{ \sum_{j=1}^\infty \frac{\mm(B_{r_j}(x_j))}{r_j}:\,E \subset \bigcup_{j=1}^\infty B_{r_j}(x_j),\,\sup_j r_j <\delta \right\}.
\end{equation}

    The codimension-$1$ Hausdorff measure $\codH{1}$ is the Borel regular outer measure defined as
\begin{equation}
    \codH{1}(E) :=\sup_{\delta >0} \codH{1}_\delta(E) \,\text{for every }E\subseteq \X.
\end{equation}
We use the notation $\codH{1}_\mm$ in place of $\codH{1}$ when we want to emphasize the measure used in the definition.
We point out that if the metric measure space is $Q$-Ahlfors regular for some $Q \ge 1$, this measure is comparable to $\mathcal{H}^{Q-1}$. The definition may be adapted to define the codimension-$p$ Hausdorff measure for $p>1$, but this is out of the scope of the note.

A third way to measure the energy of a set is the Minkowski content.
\begin{definition}
    Let $(\X,\sfd,\mm)$ be a metric measure space and let $A\subseteq \X$ be Borel. The Minkowski content of $A \subset \X$ is
\begin{equation*}
    \mm^{+}(A):=\limi_{r \to 0} \frac{\mm\left(B_r(A) \setminus A \right)}{r}.
\end{equation*}
\end{definition}

Because of some ambiguity about the Minkowski content, we define also the following quantity:
\begin{equation*}
    \widetilde{\mm^+}(A):=\limi_{r \to 0} \frac{\mm(B_r(A))}{2r}.
\end{equation*}

The perimeter, the codimension-$1$ Hausdorff measure and the Minkowski content are all notions that allow to measure the `surface area' of a Borel set in a measure-theoretic sense and they are related to each other (see for instance \cite[Section 3]{CaputoCavallucci2024}).

Moreover, they all satisfy a coarea inequality for Lipschitz functions, where the modulus of differential of a Lipschitz function in the Euclidean formula is replaced here by the local Lipschitz constant.
For the statements about the coarea inequalities, see \cite{Mir03} for the perimeter, \cite[Proposition 5.1]{AmbDiMarGig17} (and \cite[Proposition 3.7]{CaputoCavallucci2024}) for the codimension-$1$ Hausdorff measure and \cite[Prop.\ 3.5]{KorteLahti2014} and \cite[Lemma 3.2]{AmbDiMarGig17} for the Minkowski content. 

\subsection{Energy of separating sets}
\label{sec:separating_sets}

We review the characterization developed in \cite{CaputoCavallucci2024}. We only focus on some conditions presented there. We need to introduce the notion of separating sets, that may be thought as a continuous analog to the cuts in a graph, as discussed in item 2) below the statement of Theorem \ref{thm:characterization_1}.

\begin{definition}
\label{def:separating_sets}
    Let $(\X,\sfd,\mm)$ be a metric measure space and let $x,y\in \X$. A closed set $\Omega$ is a \emph{separating set from $x$ to $y$} if there exists $r>0$ such that $B_r(x) \subseteq \Omega$ and $B_r(y) \subseteq \Omega^c$. We denote by $\SS_{\textup{top}}(x,y)$ the class of all separating sets from $x$ to $y$.
\end{definition}

The main result of \cite{CaputoCavallucci2024} is the following one.

\begin{theorem}[\cite{CaputoCavallucci2024}]
\label{theo:main-intro-p=1-riproposed}
Let $(\X,\sfd,\mm)$ be a doubling metric measure space. Then the following conditions are quantitatively equivalent:
\begin{itemize}
    \item[(PI)] $1$-Poincar\'{e} inequality;
    \item[(BP)] $\exists c>0, L\geq 1$ such that $\int R_{x,y}^L\,\d {\rm Per}_\mm(\Omega, \cdot) \ge c$ for every $x, y \in \X$ and $\Omega \in \SS_{\textup{top}}(x,y)$;
    \item[(BP$_\textup{R}$)] $\exists c>0, L\geq 1$ such that ${\rm Per}_{\mm_{x,y}^L}(\Omega) \ge c$ for every $x, y \in \X$ and $\Omega \in \SS_{\textup{top}}(x,y)$;
    \item[(BMC)] $\exists c>0, L\geq 1$ such that $(\mm_{x,y}^L)^{+}(\Omega) \ge c$ for every $x, y \in \X$ and  $\Omega \in \SS_{\textup{top}}(x,y)$;
    \item[(BH)] $\exists c>0, L\geq 1$ such that $\int_{\partial \Omega} R_{x,y}^L\,\d \codH{1}_\mm \ge c$ for every $x, y \in \X$ and  $\Omega \in \SS_{\textup{top}}(x,y)$;
    \item[(BH$^e$)] $\exists c>0, L\geq 1$ such that $\int_{\partial^e \Omega} R_{x,y}^L\,\d \codH{1}_\mm \ge c$ for every $x, y \in \X$ and  $\Omega \in \SS_{\textup{top}}(x,y)$;
    \item[(BH$_\textup{R}$)] $\exists c>0, L\geq 1$ such that ${\codH{1}_{\mm_{x,y}^L}}(\partial \Omega) \ge c$ for every $x, y \in \X$ and   $\Omega \in \SS_{\textup{top}}(x,y)$;
    \item[(BH$^e_\textup{R}$)] $\exists c>0, L\geq 1$ such that ${\codH{1}_{\mm_{x,y}^L}}(\partial^e \Omega) \ge c$ for every $x, y \in \X$ and   $\Omega \in \SS_{\textup{top}}(x,y)$.
\end{itemize}
\end{theorem}

There are other conditions involving the notion of capacity and approximate modulus of family of paths, but for the sake of brevity we avoid their treatment.

Our conditions are related to previous works on ${\sf PI}$ spaces:

\begin{itemize}
    \item Our condition (BMC) is related to the main proof in \cite{FasslerOrponen19}. As we explained in item 2) after the statement of Theorem \ref{thm:characterization_1}, the proof in \cite{FasslerOrponen19} relies on a duality on a combinatorial level. In the discrete setting, this is related to the relation between capacity of a cut and flows on the graph. In the continuous level, the flow on a graph has the natural counterpart with the notion of pencil of curves, while we give a continuous of capacity of a cut, that is the Minkowski content of a separating set.

    \begin{center}
    \begin{tabular}{ | m{8em} | m{4cm}| m{4cm} | } 
    \hline
    & & \\
    Graph & Cuts in a graph & Flows in a graph \\ 
    & & \\
    \hline
    & & \\
    Metric space & Separating sets & Pencil of curves \\ 
    & & \\
    \hline
    \end{tabular}
    \end{center}
    
    \item If $(\X,\sfd,\mm)$ is $s$-Ahlfors regular, with $s \ge 1$, the condition (BH) can be easily interpreted from a more classical measure-theoretic point of view. Here, it implies that for every $x,y \in \X$ and every $\Omega \in {\rm SS}_{\rm top}(x,y)$, $\mathcal{H}^{s-1}(\partial \Omega) >0$. Thus in particular the Hausdorff dimension of $\partial \Omega$ is at least $s-1$. In other words, the Poincar\'{e} inequality implies that the boundary of separating sets has Hausdorff dimension at least $s-1$. This is also a consequence of the relative isoperimetric inequality, that will be discussed in the next section. In this result, the lower bound on the codimension-1 Hausdorff measure $\codH{1}_\mm$ of the measure-theoretic boundary of the separating sets depends on the points $x,y$. Instead, measuring the boundary of separating set with the codimension-1 Hausdorff measure weighted with the Riesz potential gives a lower bound that is independent of $x,y$.

\end{itemize}

The proof of our results relies on a combination of many preliminary results in metric analysis. We can explain in informal terms how the PI condition implies all the condition listed in our characterization in Theorem \ref{theo:main-intro-p=1-riproposed}. The characterization in Proposition \ref{prop:Poincaré_equivalences_pointwise} gives (PI) is equivalent to (PtPI). If one is allowed in the (PtPI) to take as $u$ the indicator function of a separating set, up to clarifying what the right-hand side means, one would be able to conclude.

On the other hand, each of the condition in the Theorem implies the (PI) condition by means of coarea-type formulas. Indeed, given a Lipschitz function $u$ and two points $x,y \in \X$ with $u(x)<u(y)$
\begin{equation*}
    \left\{ u >t \right\} \in \SS_{\textup{top}}(x,y).
\end{equation*}

Instead of giving an idea of the proofs of all the implications in the Theorem, we focus on the characterization (PI) $\Leftrightarrow$ (\text{BMC}) in Section \ref{sec:a_more_geometric_characterization}.

\subsection{A proof in \texorpdfstring{$\mathbb{R}^d$}{Rd} via separating sets}

\label{sec:proof_euclidean_case}

The main reason behind our investigation in \cite{CaputoCavallucci2024} is to look for new examples of ${\sf PI}$ spaces. This is the content of a forthcoming project \cite{CaputoCavallucciWald2025}. In this work, we can give an alternative proof that the first Heisenberg group $\mathbb{H}^1$, endowed with the measure $\mathcal{H}^4$, is a ${\sf PI}$ space using potential-theoretic techniques. The same techniques can be easily adapted in the Euclidean case and we provide the detailed proof here in this simplified case.

We consider as a toy model the Euclidean space $\mathbb{R}^d$ with $d \ge 2$ (the arguments can be easily adapted to $d=1$). We associate the metric measure space $(\mathbb{R}^d,|\cdot|,\mathcal{L}^d)$, where $|\cdot|$ is the Euclidean distance and $\mathcal{L}^d$ is the $d$-dimensional Lebesgue measure. We prove that $\mathbb{R}^d$ is a ${\sf PI}$ space using the condition on separating sets. 
As we previously saw, in this specific case, the Riesz kernel takes the form
\begin{equation*}
    R_x(z)=\omega_d^{-1} |x-z|^{1-d}.
\end{equation*}

Morever, we consider the Green function of the Laplacian in $\mathbb{R}^d$ with pole at $x$
\begin{equation*}
    G_x(z):=\begin{cases}
    -\frac{1}{2\pi}\ln(|x-z|) &\text{if }d=2,\\
    \frac{1}{d(d-2)\omega_d} |x-z|^{2-d}&\text{if }d \ge 3.
    \end{cases}
\end{equation*}
which satisfies $-\Delta G_x = \delta_x$ in the sense of distribution. 
Moreover, $G_x$ and $R_x$ are related by the following identity
\begin{equation}
\label{eq:gradient_green_equal_to_riesz}
    |\nabla G_x|(z)= d^{-1} R_x(z)\quad\text{for every }z \in \mathbb{R}^d.
\end{equation}

We prove that there exist $c >0$ and $L \ge 1$ such that  $\int R_{x,y}^L\,\d {\rm Per}(\Omega, \cdot) \ge c$ for every couple of points $x,y \in \X$, $x\neq y$ and $\Omega \in \SS_{\textup{top}}(x,y)$. By Theorem \ref{theo:main-intro-p=1-riproposed}, this implies that $\mathbb{R}^d$ is a ${\sf PI}$ space.

Let us fix $x,y \in \X$ and $\Omega\in \SS_{\textup{top}}(x,y)$.
If $\int R_{x,y}^L\,\d {\rm Per}(\Omega, \cdot)=\infty$, there is nothing to prove, so we may assume that $\int R_{x,y}^L\,\d {\rm Per}(\Omega, \cdot)<\infty$. In such a case, it is not hard to prove ${\rm Per}(\Omega, B_{2L\sfd(x,y)}(x))<\infty$.

We make a further semplification only for the sake of this presentation. We assume that $\Omega$ is an open set with smooth boundary. The general case of sets of finite perimeter only has the more technical complications of dealing with the notion of normal to the boundary and the more general Gauss-Green formula that holds for this class of sets. 

\begin{figure}
    \includegraphics[scale=0.4]{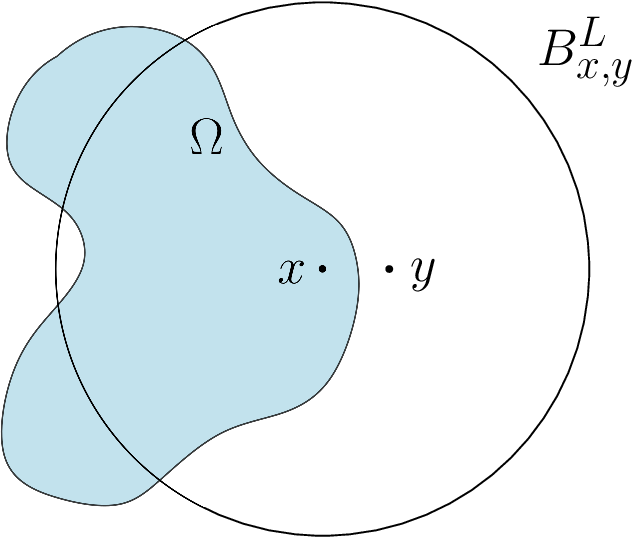}
    \caption{Example of the setting.}
\end{figure}

By the explicit formula of $R_x$ in the Euclidean space, we have that there exists a constant $c_0>0$ such that 
\begin{equation}
\label{eq:constant_energy_of_balls}
    d \int_{\partial B_r(p)} \langle \nabla G_p,\nu_{B_r(p)}\rangle\,\d \mathcal{H}^{d-1}=\int_{\partial B_r(p)} R_p\,\d \mathcal{H}^{d-1}=c_0\qquad \text{for every }p \in \X,\text{ and }r >0.
\end{equation}
Here $\nu_A(x)$ is the unit outer normal to the boundary of $A$ at $x \in \partial A$ and $A$ is an open set with smooth boundary.
Morever, given $B_r(x) \subset \Omega$, we have
\begin{equation}
\label{eq:other_set_larger_energy_than_ball}
    \int_{\partial (\Omega \cap B_{x,y}^L)} R_x\,\d \mathcal{H}^{d-1} \ge \int_{\partial B_r(x)} R_x \,\d \mathcal{H}^{d-1}=c_0.
\end{equation}
This follows from the following computation, based on \eqref{eq:gradient_green_equal_to_riesz}
\begin{equation*}
\begin{aligned}
    0 &= \int_{\Omega \cap B_{x,y}^L \setminus \overline{B_r(x)}} \Delta G_x\,\d \mathcal{L}^d = \int_{\partial (\Omega \cap B_{x,y}^L)} \langle \nabla G_x, \nu_{\partial \Omega}\rangle\,\d \mathcal{H}^{d-1}-\int_{\partial B_r(x)} \langle \nabla G_x, \nu_{B_r(x)} \rangle\,\d \mathcal{H}^{d-1}\\
    & \stackrel{\eqref{eq:constant_energy_of_balls}}{\le} \int_{\partial (\Omega \cap B_{x,y}^L)} |\nabla G_x|\,\d \mathcal{H}^{d-1}-d^{-1}\int_{\partial B_r(x)} R_x \,\d \mathcal{H}^{d-1} \\
    &\stackrel{\eqref{eq:gradient_green_equal_to_riesz}}{=} d^{-1}\int_{\partial (\Omega \cap B_{x,y}^L)} R_x\,\d \mathcal{H}^{d-1} - d^{-1}\int_{\partial B_r(x)} R_x \,\d \mathcal{H}^{d-1}.
\end{aligned}
\end{equation*}

Notice that, for every $L>0$, there exists $\delta_L$, independent of $x,y$, such that 
\begin{equation}
\label{eq:smallness_Riesz_change_pole}
    \int_{\partial B_{x,y}^L} |R_y-R_x|\,\d \mathcal{H}^{d-1} \le \delta_L.
\end{equation}
Moreover, $\delta_L \to 0^+$ as $L \to \infty$.

We can now compute
\begin{equation*}
\begin{aligned}
    &\int R_{x,y}^L\,\d {\rm Per}(\Omega,\cdot)= \int_{\partial \Omega \cap B_{x,y}^L} R_x\,\d \mathcal{H}^{d-1} + \int_{\partial \Omega \cap B_{x,y}^L} R_y\,\d \mathcal{H}^{d-1}\\
    & = \int_{\partial (\Omega \cap B_{x,y}^L)} R_x\,\d \mathcal{H}^{d-1} - \int_{\Omega \cap \partial B_{x,y}^L} R_x\,\d \mathcal{H}^{d-1} + \int_{\partial (\Omega^c \cap B_{x,y}^L)} R_y\,\d \mathcal{H}^{d-1} - \int_{\Omega^c \cap \partial B_{x,y}^L} R_y\,\d \mathcal{H}^{d-1}\\
    &\stackrel{\eqref{eq:other_set_larger_energy_than_ball}}{\ge} 2c_0- \int_{\Omega \cap \partial B_{x,y}^L} R_x\,\d \mathcal{H}^{d-1} - \int_{\Omega^c \cap \partial B_{x,y}^L} R_y\,\d \mathcal{H}^{d-1} \stackrel{\eqref{eq:smallness_Riesz_change_pole}}{\ge} 2c_0- \int_{\partial B_{x,y}^L} R_x\,\d \mathcal{H}^{d-1} - \delta_L\\
    & =c_0-\delta_L.
\end{aligned}
\end{equation*}

Choosing $L$ sufficiently large, we have that $\delta_L < c_0/2$, thus $\int R_{x,y}^L\,\d {\rm Per}(\Omega,\cdot) \ge c_0/2$, thus concluding.

\subsection{Relative isoperimetric inequality}

The Poincar\'{e} inequality is related to the validity of the relative isoperimetric inequality. This is due to different contributions \cite{Mir03,Amb02,KorteLahti2014,Lahti2020}, that we briefly review in this section.

In informal terms, the relative isoperimetric inequality says the following. Fix a set function $\sigma$ which `measures the surface measure' of a given set. There exists a constant $C>0$ such that for every Borel set $E$ and every ball $B_{r}(x)$
\begin{equation*}
    \frac{\min\left\{ \mm(E \cap B_r(x)), \mm(E^c \cap B_r(x)) \right\}}{\mm(B_r(x))} \le C r \frac{\sigma(\partial E \cap B_{\lambda r}(x))}{\mm(B_{\lambda r}(x))}.
\end{equation*}
We review characterizations of ${\sf PI}$ spaces in terms of relative isoperimetric inequalities where the function $\sigma$ varies among the energies defined at the beginning of Section \ref{sec:codimension_1}.

Miranda in \cite{Mir03} (see also \cite{Amb02}) proved that a necessary condition in a ${\sf PI}$ space is the validity of the following relative isoperimetric inequality. There exists a constant $C >0$ such that for all Borel sets $E\subset\X$ and every ball $B_r(x) \subset \X$
\begin{equation*}
    \frac{\min\left\{ \mm(E \cap B_r(x)), \mm(E^c \cap B_r(x)) \right\}}{\mm(B_r(x))} \le C r \frac{{\rm Per}(E, B_{\lambda r}(x))}{\mm(B_{\lambda r}(x))}.
\end{equation*}
This is a consequence of the fact that the definition of ${\sf PI}$ spaces can be equivalently formulated with BV functions $u$ replacing the class of BV functions and with $|D u|(B_{\lambda r}(x))/\mm(B_{\lambda r}(x))$ in the inequality replacing the averaged integral of $\lip u$. Then it is enough to test the inequality with the indicator of a set of finite perimeter. A crucial estimate is the simple algebraic computation
\begin{equation}
\begin{aligned}
\label{eq:equivalence_lefthandside}
    \frac{1}{2} \frac{\min \left\{ \mm(E \cap B_r(x)),\mm(E \setminus B_r(x)) \right\}}{\mm(B_r(x))} &\le \dashint_{B_r(x)} \left|\chi_E -\dashint_{B_r(x)} \chi_E\,\d \mm \right|\,\d \mm\\
    &\le 2 \frac{\min \left\{ \mm(E \cap B_r(x)),\mm(E \setminus B_r(x)) \right\}}{\mm(B_r(x))}.
\end{aligned}
\end{equation}

The sufficiency of the validity of the above relative isoperimetric inequality to get that $\X$ is a ${\sf PI}$ space follows by applying the inequality to the superlevel sets of a Lipschitz function $u$, say $E_t:=\{ u>t \}$. Then, as proven in \cite{KorteLahti2014}, one integrate with respect to the $t$ variable using coarea formula for BV functions (\cite[Proposition 4.2]{Mir03}) and \eqref{eq:equivalence_lefthandside}.

However, one may consider other energies to measure the boundary, as the codimension-1 Hausdorff measure and the Minkowksi content. This leads to the following definitions {\cite[Section 1]{KorteLahti2014}}.

    \begin{itemize}
        \item[({\rm IsoM})] there exist $C>0$, $\lambda \ge 1$ such that
        \begin{equation*}
          \frac{\min \left\{ \mm(B_r(x) \cap E),\mm(B_r(x)\setminus E)\right\}}{\mm(B_r(x))} \le C r \frac{\widetilde{\mm^+}(B_{\lambda r}(x) \cap \partial E)}{\mm(B_{\lambda r}(x))}   
        \end{equation*}
        for all $\mm$-measurable set $E$ and $x \in \X$, $r>0$;
        \item[({\rm IsoM$^e$})] there exist $C>0$, $\lambda \ge 1$ such that
        \begin{equation*}
          \frac{\min \left\{ \mm(B_r(x) \cap E),\mm(B_r(x)\setminus E)\right\}}{\mm(B_r(x))} \le C r \frac{\widetilde{\mm^+}(B_{\lambda r}(x) \cap \partial^e E)}{\mm(B_{\lambda r}(x))}   
        \end{equation*}
        for all $\mm$-measurable set $E$ and $x \in \X$, $r>0$;
        \item[({\rm IsoH})] there exist $C>0$, $\lambda \ge 1$ such that
        \begin{equation*}
          \frac{\min \left\{ \mm(B_r(x) \cap E),\mm(B_r(x)\setminus E)\right\}}{\mm(B_r(x))} \le C r \frac{\codH{1}(B_{\lambda r}(x) \cap \partial E)}{\mm(B_{\lambda r}(x))}   
        \end{equation*}
        for all $\mm$-measurable set $E$ and $x \in \X$, $r>0$.
    \end{itemize}

These conditions characterize ${\sf PI}$ spaces. 

\begin{proposition}[\cite{KorteLahti2014}]
    Let $(\X,\sfd,\mm)$ be a doubling metric measure space. Then $\X$ is a ${\sf PI}$ space if one (and thus all) of the conditions ({\rm IsoM}), ({\rm IsoM$^e$}), ({\rm IsoH}) holds.
\end{proposition}

\begin{proof}
We sketch the proof. ({\rm IsoM$^e$}) implies ({\rm IsoM}) because $\partial^e E\subset \partial E$.
({\rm IsoH}) implies ({\rm IsoM}) because $\codH{1}(A)\le C \widetilde{\mm^+}(A)$ for all $A \subset \X$. ({\rm IsoM}) implies that $\X$ is a ${\sf PI}$ space by arguing as in the case of the perimeter, by integrating the condition on the superlevel set of a Lipschitz function. In this case, one applies the coarea inequality for the Minkowski content.

The condition ({\rm IsoM}) implies ({\rm IsoM$^e$}) via the following more delicate argument. The terms in both sides in ({\rm IsoM$^e$}) remains unchanged if we change $E$ with $\tilde{E}$ such that $\mm(E \Delta \tilde{E})=0$. Selecting as $\tilde{E}$ the measure-theoretic interior of $E$, one has that $\partial \tilde{E}=\overline{\partial^e E}$. By the very definition of $\widetilde{\mm^+}$, we have
\begin{equation*}
    \widetilde{\mm^+}(B_{\lambda r}(x) \cap \partial \tilde{E}) = \widetilde{\mm^+}(B_{\lambda r}(x) \cap \partial^e E)
\end{equation*}
which, together with the assumption, gives the conclusion (see \cite[Theorem 3.6]{KorteLahti2014}).

The (PI) condition implies ({\rm IsoH}) by the same classical argument of analysis on metric spaces used in item i) below the statement of Theorem \ref{thm:characterization_1}, that in this case reads as follows. Given a Borel function $g$, we find a Lipschitz function $u$ such that $g$ is its upper gradient. In this case, the $L^1$ norm of $g$ is an approximation (up to a multiplicative constant) of $\codH{1}(B_{\lambda r}(x) \cap \partial E)$ and $u$ is an approximation of $\chi_E$. For more details, we refer the reader to \cite[Theorem 3.13]{KorteLahti2014}.
\end{proof}

An expert reader in geometric measure theory will notice that ({\rm IsoH}) is not the most natural condition involving the codimension-1 Hausdorff measure.

This is a consequence of the celebrated De Giorgi-Federer characterization of sets of finite perimeter in the Euclidean space. Indeed, Federer (see \cite[Section 4.5.11]{Federer69}), after De Giorgi, gave the following characterization in $\mathbb{R}^d$. Given an open set $\Omega \subset \mathbb{R}^d$ and an $\mathcal{L}^d$-measurable set $E\subset \mathbb{R}^d$ of finite perimeter, we have that ${\rm Per}(E,\Omega)< \infty$ if and only if $\mathcal{H}^{d-1}(\partial^e E \cap \Omega)< \infty$.

Since $\codH{1}$ is the natural replacement of $\mathcal{H}^{d-1}$ that takes into account the local change of dimension of the metric space, a natural formulation of the relative isoperimetric inequality is given by the following condition.
\begin{itemize}
    \item[({\rm IsoH$^e$})] there exist $C>0$, $\lambda \ge 1$ such that
        \begin{equation*}
          \frac{\min \left\{ \mm(B_r(x) \cap E),\mm(B_r(x)\setminus E)\right\}}{\mm(B_r(x))} \le C r \frac{\codH{1}(B_{\lambda r}(x) \cap \partial^e E)}{\mm(B_{\lambda r}(x))}   
        \end{equation*}
        for all $\mm$-measurable set $E$ and $x \in \X$, $r>0$.
\end{itemize}
Lahti proved in \cite[Corollary 5.4]{Lahti2020} that ({\rm IsoH$^e$}) is equivalent to the fact that $\X$ is a ${\sf PI}$ space. The main tool is the validity of a De Giorgi-Federer characterization of sets of finite perimeter in metric spaces, as proved in \cite{Lahti2020}, building upon preliminary tools in \cite{Lahti17}. It takes the following form. Let $(\X,\sfd,\mm)$ be a ${\sf PI}$ space. Let $\Omega \subset \X$ be an open set and let $E \subset \X$ be $\mm$-measurable. Then, ${\rm Per}(E,\Omega)<\infty$ if and only if $\codH{1}(\partial^e E \cap \Omega)< \infty$.

\section{A more geometric relation between obstacle-avoidance principle and energy of separating sets}
\label{sec:a_more_geometric_characterization}

The goal of this section is twofold. We show a proof of Proposition \ref{prop:characterization_3}. The proof uses as a tool the (BMC) condition introduced Section \ref{sec:separating_sets} and, indeed, we also prove in the same proof that (PI) is equivalent to (BMC) (an equivalence of Theorem \ref{theo:main-intro-p=1-riproposed}). The second aim is to directly relate (BMC) to 1-set-connectedness. This will be proved in Theorem \ref{thm:main_theorem_CC2}. The main tool for this second goal is the position function, that we define and explain in Section \ref{sec:position_function}.

We define the following auxiliary quantity.

$$\SR_{x,y}(A) := \frac{\mm_{x,y}^L(A)}{\width_{x,y}(A)}.$$

The $1$-set-connectedness can be recast into the following condition. There exist constants $c>0$, $L \ge 1$ such that 
\begin{equation*}
    \SR_{x,y}(A) \ge c
\end{equation*}
for every Borel set $A \in \X$.

The first result of this section is the relate the (PI) condition, the obstacle-avoidance condition with the Riesz kernel and the Minkowski content of separating sets. 

\begin{theorem}[{\cite[Theorem 1.4]{CaputoCavallucci2024II}}]
    \label{theo:main-intro-p=1}
    Let $(\X,\sfd,\mm)$ be a doubling metric measure. Then the following conditions are quantitatively equivalent:
    \begin{itemize}
        \item[(i)] $(\X,\sfd,\mm)$ is a ${\sf PI}$ space;
        \item[(ii)] the space $(\X,\sfd,\mm)$ is $(C,L)$ $1$-set-connected; 
        \item[(iii)] there exist $C,L$ such that \eqref{eq:defin_set_connectedness_bounded_intro} is satisfied by every closed subset $A\subseteq \X$;
        \item[(iv)] the space $(\X,\sfd,\mm)$ satisfies \textup{(BMC)}.
    \end{itemize}

\end{theorem}

\begin{proof}
    Condition (i) implies that for every $x,y\in \X$ there exists a pencil of curves $\alpha \in \mathcal{P}(\Gamma_{x,y}^L)$, which gives
    \begin{equation*}
        \width_{x,y}(A)\le\inf_{\gamma \in \Gamma_{x,y}^L} \ell(\gamma \cap A) \le \int \ell(\gamma \cap A)\,\d \alpha(\gamma) \le C \mm_{x,y}^L(A)
    \end{equation*}
    for every Borel set $A \subset \X$. This proves (ii). (ii) implies (iii) is trivial. Now, we prove that (iii) implies (iv). We take a separating set $\Omega \in \SS_{\textup{top}}(x,y)$. If $(\mm_{x,y}^L)^+(\Omega) = +\infty$ there is nothing to prove, so we suppose $(\mm_{x,y}^L)^+(\Omega) < +\infty$, which implies $\mm_{x,y}^L(\partial \Omega) = 0$. Consider the set $A_r:= \overline{B}_r(\Omega) \setminus \Int(\Omega)$, which is a closed subset of $\X$. We have that
    \begin{equation*}
        \width_{x,y}(A_r) \geq r\text{ if }r < \min\{\sfd(\partial \Omega, x), \sfd(\partial \Omega, y) \}
    \end{equation*}
    Indeed, because of this choice of $r$, every curve has to travel from the interior of $\Omega$ to the complement of $\overline{B}_r(\Omega)$. Thus, the length of such curve is at least $r$.
    We can now compute
     $$(\mm_{x,y}^L)^+(\Omega) = \limi_{r \to 0}\frac{\mm_{x,y}^L(\overline{B}_r(\Omega) \setminus \Omega)}{r} = \limi_{r \to 0}\frac{\mm_{x,y}^L(A_r)}{r} \geq \limi_{r \to 0}\frac{\mm_{x,y}^L(A_r)}{\width_{x,y}(A_r)} \geq \inf_{\substack{A \subseteq \X \\ A \text{ closed}}} \SR_{x,y}(A) \geq c,$$
     where we use the assumption in the last inequality.
     Since this is true for every $\Omega \in \SS_{\textup{top}}(x,y)$ we get (iv).
     
     To prove that (iv) implies (i), one repeats the argument of Theorem \ref{theo:main-intro-p=1-riproposed}. More precisely, given a Lipschitz function $u$ and a pair points $x,y\in \X$, we have that the sets $\Omega_t := \lbrace u \geq t \rbrace$ belong to $\SS_\text{top}(x,y)$ for all $t\in (u(x),u(y))$. In doing this, we assumed without loss of generality that $u(x)<u(y)$. We can apply the coarea inequality for the Minkowski content and the measure $\mm_{x,y}^L$ to integrate the (BMC) property and conclude.
    
\end{proof}

This proof is very analytical and, in our opinion, does not show the direct relation between the obstacle-avoidance principle and the Minkowski content of separating sets, that have both a geometric definition. The main goal of \cite{CaputoCavallucci2024II} is to study more carefully this relation.

We prove that the infimum of the separating ratio among closed sets and the infimum of weighted Minkowski content of separating sets can be compared in terms of a constant which is related only the $\Lambda$-quasiconvexity of the space. The equivalence does not pass through the validity of a Poincar\'{e} inequality.

\begin{theorem}
\label{thm:main_theorem_CC2}
Let $(\X,\sfd,\mm)$ be a doubling, path-connected, locally $\Lambda$-quasiconvex metric measure space and let $x,y\in \X$. Then
\begin{equation}
    \label{eq:equivalence_of_minima_intro}
    \Lambda^{-1}\inf_{\Omega \in \SS_{\textup{top}}(x,y)} (\mm_{x,y}^L)^+(\Omega) \le \inf_{A \subseteq \X\, \textup{closed}} \SR_{x,y}(A)\le \inf_{\Omega \in \SS_{\textup{top}}(x,y)} (\mm_{x,y}^L)^+(\Omega).
\end{equation}
In particular, if $(\X,\sfd)$ is path connected and locally geodesic we have
\begin{equation}
    \label{eq:equality_of_minima_geodesic_case}
    \inf_{A \subseteq \X\, \textup{closed}} \SR_{x,y}(A)= \inf_{\Omega \in \SS_{\textup{top}}(x,y)} (\mm_{x,y}^L)^+(\Omega).
\end{equation}
\end{theorem}

We refer the reader to Section \ref{sec:position_function} for the definition of locally $\Lambda$-quasiconvexity.

\subsection{Explanation of the geometric idea behind the proof of Theorem \ref{thm:main_theorem_CC2}} 
We consider Figure \ref{fig:dumbbell} and we compute the infimum of the separating ratio between $x$ and $y$ as in the picture. The choice of the cone-shaped domain $D$ as a competitor is not convenient for the minimization of the separating ratio. Indeed, we can reduce the volume of the set (with respect to $\mm_{x,y}^L$) by keeping unaltered its width. This will create a better competitor for the infimum of the separating ratio, given by the rectangle $R$ in Figure \ref{fig:dumbbell}. The next step is to chop $R$ in $N$ slimmer rectangles $R_i$ of width equal to the $N$-th fraction of the width of $R$. We then select one of such slimmer rectangles, say $R_1$, with the property that its separating ratio is less than or equal to the separating ratio of $R$. This is a consequence of the following formula:
$$\frac{1}{N}\sum_{i=1}^N \SR_{x,y}(R_i) = \SR_{x,y}(R).$$
This formula is trivial in such a case and can be regarded as a discrete counterpart of coarea formula. 

We can repeat this procedure, thus finding thinner and thinner rectangles with decreasing separating ratio. The sequence of thin rectagles converges to a vertical line $L$. This line is the boundary of a separating set, that is the half-plane with boundary $L$ containing $x$. Its Minkowski content can be estimated from above by the separating ratio of the converging rectangles. This is the main idea in the proof of Theorem \ref{thm:main_theorem_CC2} and all the objects involved will be discussed in details in Section \ref{sec:position_function}.

\begin{figure}[h!]
    \centering
    \includegraphics[scale=0.6]{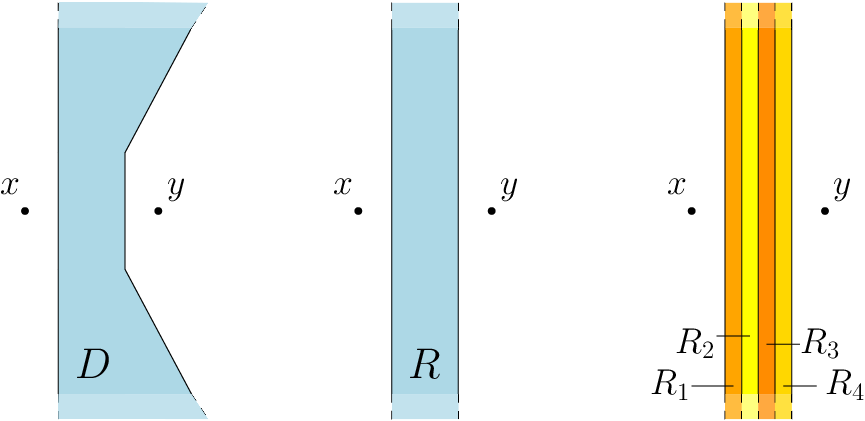}
    \caption{The picture gives an informal explanation of the proof of the Theorem in the the toy example of the two dimensional Euclidean case for a specific choice of $x,y$ and unbounded $D$ in the definition of separating ratio.}
    \label{fig:dumbbell}
\end{figure}

\subsection{The position function and the proof of Theorem \ref{thm:main_theorem_CC2}}
\label{sec:position_function}

We introduce here a function that allows to fibrate a closed set into boundaries of separating sets. This is the so-called \emph{position function} and it is defined as follows (for more details, see \cite[Section 5]{CaputoCavallucci2024II}).

Given a Borel $A \subseteq \X$ and $\gamma \in \Gamma_{x,y}$ we define the \emph{position function along the path $\gamma$ of the set $A$} as 
\begin{equation}
    \label{eq:posgamma_continuous_definition}
    \pos_{\gamma, A} \colon \X \to [0,\infty],\qquad \pos_{\gamma,A}(z):= \inf_{s \in \gamma^{-1}(z)} \ell(\gamma \cap A,s),
\end{equation}
with the usual convention that ${\rm pos}_{\gamma,A}(z)=\infty$ if $z\notin {\rm Im}(\gamma)$. In other words $\pos_{\gamma,A}(z)$ is the length spent by $\gamma$ inside $A$ before reaching for the first time the point $z$.
We define the \emph{position function with respect to the set $A$} as
\begin{equation}
\pos_A \colon \X \to [0,\infty],\qquad \pos_A(z):=\inf_{\gamma \in \Gamma_{x,y}} \pos_{\gamma,A}(z).
\end{equation}

To study the fine properties of such function, we define some more refined connectivity condition of the space:
\begin{itemize}
    \item[-] rectifiable path connected if $\Gamma_{x,y} \neq \emptyset$ for every $x,y\in \X$;
    \item[-] pointwise rectifiable path connected if for every $x\in \X$ there exists $r_x > 0$ such that $\Gamma_{x,y} \neq \emptyset$ for every $y\in B_{r_x}(x)$;
    \item[-] pointwise $\Lambda_x$-quasiconvex at $x\in \X$ if there exist $r_x > 0$ and $\Lambda_x \ge 1$ such that $\Gamma_{x,y}^{\Lambda_x} \neq \emptyset$ for every $y\in B_{r_x}(x)$;
    \item[-] pointwise quasiconvex if it is pointwise quasiconvex at every $x\in \X$;
    \item[-] locally $\Lambda$-quasiconvex if for every $x\in \X$ there exist $r_x > 0$ such that $\Gamma_{y,z}^{\Lambda} \neq \emptyset$ for every $y,z\in B_{r_x}(x)$.
\end{itemize}

To familiarize with the position function, we mention the following proposition that shows an interesting feature of it. We consider a set $A \subset \X$ and a pair of points such that it $A$ has the property that every rectifiable curve that travels from one point to the other crosses $A$. Then every level set of the position function intersected with $A$ has the same property.

\begin{proposition}[{\cite[Proposition 5.3]{CaputoCavallucci2024II}}]
\label{prop:level_set_position_function}
        Let $(\X,\sfd)$ be a metric space and let $x,y\in \X$. Let $A$ be closed such that $A \cap \gamma \neq \emptyset$ for all  $\gamma \in \Gamma_{x,y}$. Then the sets $\{z\in \X \, : \,\pos_A(z) = t \} \cap A$ have the property that $\{z\in \X \, : \,\pos_A(z) = t \} \cap A \cap \gamma \neq \emptyset$ for all  $\gamma \in \Gamma_{x,y}$ for all $t \in [0,\width_{x,y}(A)]$.
\end{proposition}

\begin{figure}[h!]
    \centering
    \includegraphics[scale=0.6]{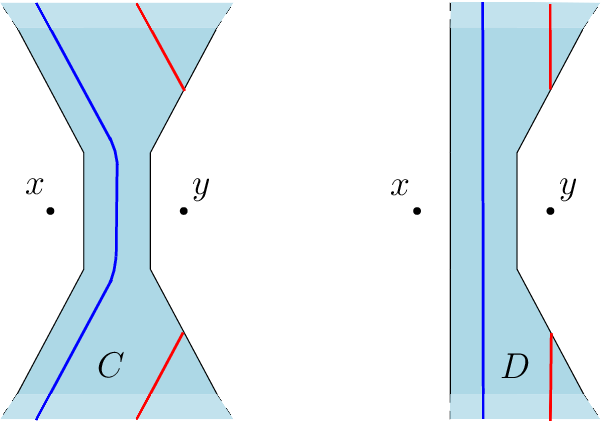}
    \caption{Fix the pair of points $x,y$ in the plane. Then consider the unbounded sets $C$ and $D$. Both of them satisfy the assumption of Proposition \ref{prop:level_set_position_function}. The blue lines represent a level set of the position function for $t \in [0,\width(A)]$, while the red one for $t > \width(A)$.}
    \label{fig:enter-label}
\end{figure}

Next, we can finely quantify the regularity of the position function in terms of the metric connectivity properties of the ambient space as follows.

\begin{proposition}[{\cite[Proposition 5.4]{CaputoCavallucci2024II}}]
\label{prop:position_local_properties}
        Let $(\X,\sfd)$ be a rectifiable path connected metric space and let $x,y\in \X$. Let $A \subseteq \X$ Borel. Then 
        \begin{itemize}
            \item[(i)] $\pos_A(z) < +\infty$ for every $z\in \X$;
            \item[(ii)] if $\X$ is pointwise $\Lambda_z$-quasiconvex at $z\in \X$ then $\lip \,\pos_A(z) \leq \Lambda_z$;
            \item[(iii)] if $\X$ is pointwise quasiconvex then $\pos_A$ is continuous and $\lip\,\pos_A = 0$ on $\overline{A}^c$;
            \item[(iv)] if $\X$ is (locally) $\Lambda$-quasiconvex then $\pos_A$ is (locally) $\Lambda$-Lipschitz.
        \end{itemize}
\end{proposition}

The last proposition is the crucial tool that we need, in combination with a coarea inequality for Lipschitz functions, to prove Theorem \ref{thm:main_theorem_CC2}.

\begin{proof}[Sketch of the proof of Theorem \ref{thm:main_theorem_CC2}]
    The second inequality can be proven as in (iii)$\Rightarrow$(iv) in Theorem \ref{theo:main-intro-p=1}.
    We sketch the first inequality. Since $\Gamma_{x,y} \neq \emptyset$, $\width_{x,y}(A) < +\infty$ for every $A\subset \X$. Fix a closed set $A \subset \X$. We can suppose $\width_{x,y}(A) > 0$, otherwise $\SR_{x,y}(A) = +\infty$ and there is nothing to prove. We consider the function $\pos_A$ that is $\Lambda$-Lipschitz on $\X$ as a combination of item (iv) in Proposition \ref{prop:position_local_properties} and the fact that ${\rm supp}(\mm_{x,y}^L)$ is compact.  

    We notice that, for $t\in (0,\width_{x,y}(A))$, the sets $\{\pos_A \le t\}$ belong to $\SS_{\text{top}}(x,y)$ because $\pos_A$ is Lipschitz, $\pos_A(x) = 0$, $\pos_A(y) = \width_{x,y}(A)$.
    
    By the coarea inequality and items (iii),(iv) in Proposition \ref{prop:position_local_properties}, we have
    \begin{equation*}
        \int_0^{\width_{x,y}(A)} (\mm_{x,y}^L)^+(\{ \pos_A \le t \}) \,\d t \le \int \lip \,\pos_A \,\d \mm_{x,y}^L = \int_A \lip \,\pos_A \,\d \mm_{x,y}^L \le \Lambda \mm_{x,y}^L(A).
    \end{equation*}
    Dividing by $\width_{x,y}(A)$ we conclude that we can find some $\bar{t}\in (0,\width_{x,y}(A))$ such that
    $$(\mm_{x,y}^L)^+(\{\pos_A\}) \le \bar{t}\}) \le \Lambda \frac{\mm_{x,y}^L(A)}{\width_{x,y}(A)} = \Lambda\cdot \SR_{x,y}(A).$$
    By the arbitrariness of $A$ and since $\{\pos_A \le \bar{t}\} \in {\rm SS}_{\rm top}(x,y)$, we finally have
    $$\frac{1}{\Lambda}\inf_{\Omega \in \SS_{\textup{top}}(x,y)} (\mm_{x,y}^L)^+(\Omega) \le \inf_{A \subseteq \X\,\textup{closed}} \frac{\mm_{x,y}^L(A)}{\width_{x,y}(A)}.$$
    
\end{proof}

\section{Open problems}
\label{sec:open}
We give some examples of open problems. One of the reason to develop all the machinery in \cite{CaputoCavallucci2024,CaputoCavallucci2024II} is to prove that two classes of metric measure spaces with bound on the curvature satisfy the Poincar\'{e} inequality.

\subsection{\texorpdfstring{${\sf MCP}(0,N)$}{MCP} spaces}
The first class is the one of metric measure spaces that satisfy a measure contraction property ${\sf MCP}(K,N)$ in the sense of Ohta-Sturm (\cite{Ohta07,SturmII}), where $K \in \mathbb{N}$ is a lower bound on the curvature and $N > 1$ is an upper bound on the dimension. This class generalizes the one of ${\sf CD}(K,N)$ spaces, but it has the positive side of including subRiemannian and subFinslerian structures, as the Heisenberg group (\cite{Jui09}), some Carnot groups (\cite{Riz16}) and the sub-Finslerian Heisenberg group (\cite{BorzaMagnaboscoRossiTashiroI}).
We restrict the study to the class ${\sf MCP}(0,N)$ space, because they are doubling metric measure spaces. One can change the axiomatization of doubling and Poincar\'{e} to local statements to study ${\sf MCP}(K,N)$ for negative $K$. 
Tapio Rajala raised the question if ${\sf MCP}(0,N)$ spaces are ${\sf PI}$ spaces.
Eriksson-Bique proved in \cite[Theorem 1.3]{ErikssonBique2019} that they satisfy a $p$-Poincar\'{e} inequality if $p > N+1$, meaning the $L^p$-norm of $\lip u$ replaces the $L^1$ norm of $\lip u$ on the right-hand side. 
It is not known whether ${\sf MCP}(0,N)$ are ${\sf PI}$ spaces, so it is not clear if they satisfy a $1$-Poincar\'{e} inequality.

\subsection{ \texorpdfstring{${\sf GCBA}(\kappa)$}{GCBA} spaces}
The second class is that of metric spaces with an upper bound on the curvature. More specifically, Lytchak and Nagano in \cite{LN19} studied the local geometry of the class of complete and separable, locally compact, metric space such that
\begin{itemize}
    \item every geodesic can be extended to a geodesic defined on a larger interval. This metric property is called \emph{geodesic completeness};
    \item the curvature is bounded from above by $\kappa \in \mathbb{R}$ in the following sense. Triangles in the metric space are thinner than the triangles with same edge length in the two dimensional model space of constant sectional curvature $\kappa$.
\end{itemize}

These are called called ${\sf GCBA}(\kappa)$ spaces, that stands for Geodesically complete spaces with Curvature Bounded Above.
These spaces are natural generalizations of smooth Riemannian manifolds (with no boundary) with curvature bounded above.

At this level of generality, a space in this class does not have a fixed dimension; indeed, one can define simple examples of varying dimension by gluing a segment with a sphere. Nevertheless, we restrict to the case of $n$-dimensional spaces for $n \in \mathbb{N}$, that means that every point has a small ball around it of Hausdorff dimension $n$. The $n$-dimensional Hausdorff measure on a $n$-dimensional ${\sf GCBA}$ space is (locally) doubling, as proved in \cite{CavSam22}. 
It is natural to ask whether, after the previous restriction, a $n$-dimensional ${\sf GCBA}$ space is a ${\sf PI}$ space.
This is not the case as the following simple examples show. A first observation is that the ${\sf GCBA}$ is stable under a gluing procedure. Given two ${\sf GCBA}(\kappa)$ spaces $\X$ and $\Y$ and a geodesically convex $A\subset \X$ that is isometric to $\tilde{A}\subset \Y$. The gluing of such spaces along $A$ is still a ${\sf GCBA}(\kappa)$ space.

Let $d\in \mathbb{N}$ and consider two copies of the Euclidean space $\mathbb{R}^d$, which are ${\sf GCBA}(0)$ spaces. For $1\le k < d$, we glue them along a convex $k$-dimensional set (for $k=1$ one can consider a segment, for $k=2$ a $2$-dimensional square and so on).
The resulting space satisfying a $p$-Poincar\'{e} inequality if and only if $p > d-k$.

We refer the reader to \cite[Section 6.14]{HeiKos98} for the discussion about the definition of a new metric measure space with a gluing procedure and the validity of the Poincar\'{e} inequality on the glued space.

\begin{figure}[h!]
        \label{fig:quotient_R2}
        \includegraphics[scale=0.7]{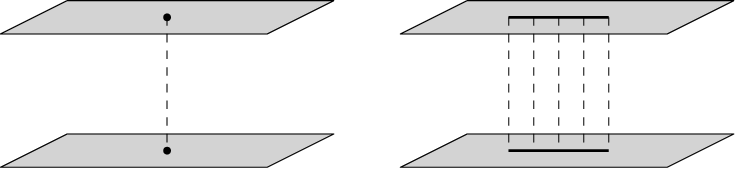}
        \caption{Example of gluings in the 2-dimensional case.}
\end{figure}

Therefore, a more natural question would be to find a geometric condition on a $n$-dimensional ${\sf GCBA}$ space which characterizes the validity of a $1$-Poincar\'{e} inequality.

\bibliographystyle{alpha}
\bibliography{biblio.bib}
\end{document}